\documentclass[11pt]{amsart}

\usepackage{mathtools}
\usepackage{amssymb}
\usepackage{enumitem}
\usepackage[colorlinks=true,linkcolor=blue,citecolor=blue,urlcolor=blue]{hyperref}

\setlist{nosep}
\numberwithin{equation}{section}

\newcommand{\D}{\mathbb D}

\newcommand{\Nzero}{\mathbb N_0}
\newcommand{\A}{A^2(\D)}
\newcommand{\Bop}{\mathcal B}
\newcommand{\Kop}{\mathcal K}
\newcommand{\Top}{\mathcal T}
\newcommand{\Sone}{\mathcal S_1}
\newcommand{\Tr}{\operatorname{Tr}}

\newcommand{\HS}{\mathrm{HS}}
\newcommand{\id}{\mathrm{Id}}

\newtheorem{theorem}{Theorem}[section]
\newtheorem{lemma}[theorem]{Lemma}
\newtheorem{corollary}[theorem]{Corollary}
\newtheorem{proposition}[theorem]{Proposition}
\theoremstyle{remark}
\newtheorem{remark}[theorem]{Remark}

\title[A compact counterexample to Berezin approximation]
{A Compact Counterexample to Su\'arez's Berezin Approximation Question}

\author{Mishko Mitkovski}
\address{School of Mathematical and Statistical Sciences, Clemson University,
Clemson, SC 29630}
\email{mmitkov@clemson.edu}

\author{Yuyuan Ouyang}
\address{School of Mathematical and Statistical Sciences, Clemson University,
Clemson, SC 29630}
\email{yuyuano@clemson.edu}

\date{August 2026}

\subjclass[2020]{Primary 47B35; Secondary 32A36, 47A30, 47L80}
\keywords{Bergman space, Berezin transform, Toeplitz algebra, compact operator,
Hankel matrix, trace class}

\begin{document}

\begin{abstract}
Let $A^2(\D)$ be the unweighted Bergman space and write
$Q_m(S)=T_{B_m(S)}$ for the map induced by the $m$th higher-order
Berezin transform.  Su\'arez asked in~\cite{{Suarez2005}} whether $Q_m(S)$ converges to $S$ in
operator norm for every $S$ in the full Bergman Toeplitz algebra. We answer
this question negatively in a strong form: there is a compact operator
$S$ such that
\[
   \sup_{m\geq 0}\|Q_m(S)\|=\infty.
\]
In particular, along a strictly increasing sequence $(m_n)$ one has
\[
   \|Q_{m_n}(S)-S\|\longrightarrow\infty.
\]

The obstruction is a moving matrix edge. The proof uses a moving family of
rank-one test operators, each supported in the \(m\)th matrix column. Near
the corresponding edge, these test operators are sent to weighted Hankel
matrices, and the limiting coefficients form an explicit Pascal kernel. We prove a standalone
Pascal--Hankel theorem showing that the associated weighted Hankel
transformation fails to map $\ell^2$ boundedly into trace class.  Finite-section
convergence, trace duality, and the Uniform Boundedness Principle then transfer
this instability to the exact Bergman maps.
\end{abstract}

\maketitle

\section{Introduction}

Let $dA=\pi^{-1}\,dx\,dy$, and let $A^2(\D)$ denote the Bergman subspace of
analytic function in $L^2(\D,dA)$.  We use the orthonormal basis
\[
   e_n(z)=\sqrt{n+1}\,z^n,
\]
of $A^2(\D)$ and write $E_{pq}=e_p\otimes e_q$ for the matrix units operators on $A^2(\D)$, where
$(x\otimes y)f=\langle f,y\rangle x$.  For $a\in L^\infty(\D)$, let
$T_a f=P(af)$ be the Toeplitz operator and set
\[
   \Top(L^\infty)=C^*(T_a:a\in L^\infty(\D)).
\]
For $m\in\Nzero$, let $B_m$ be the $m$th higher-order Berezin transform and
put $Q_m(S)=T_{B_m(S)}$.

Su\'arez proved norm convergence $Q_m(T_a)\to T_a$ for
$a\in L^\infty(\D)$ and, more generally, established convergence for several
important classes of operators.  He then asked whether
\begin{equation}\label{eq:suarez-question}
   \|Q_m(S)-S\|\longrightarrow 0
   \qquad\text{for every }S\in\Top(L^\infty);
\end{equation}
see \cite[Problem~(2), p.~20]{Suarez2005}.  He indicated that he expected an
affirmative answer.  Further positive evidence comes from the radial theory:
every radial element of the Toeplitz algebra is approximated in norm by its
higher-order Berezin--Toeplitz operators
\cite{BauerHerreraVasilevski,DawsonDewageMitkovskiOlafsson,Suarez2005}. It is important that the individual Toeplitz operators are norm dense in
\(\mathcal T(L^\infty)\)~\cite{Xia}. Nevertheless, convergence on this dense class does
not pass to its norm closure without uniform bounds for the family
\((Q_m)_{m\geq0}\).

Our result shows that the passage from individual Toeplitz operators and
radial elements to the full Toeplitz algebra is impossible.  In fact, the failure already occurs for a compact input.

\begin{theorem}[Main theorem]\label{thm:main}
There exists $S\in\Kop(\A)$ such that
\[
   \sup_{m\geq 0}\|Q_m(S)\|=\infty.
\]
Consequently, there are strictly increasing integers $m_n\to\infty$ for which
\[
   \|Q_{m_n}(S)\|\geq 2^n
   \quad\text{and}\quad
   \|Q_{m_n}(S)-S\|\longrightarrow\infty.
\]
In particular, $S$ is a counterexample to \eqref{eq:suarez-question}.
\end{theorem}

\begin{corollary} One has
\[
\mathcal T(L^\infty(\mathbb D))
\neq
L^\infty(\mathbb D)*_\pi\mathcal S_1(A^2(\mathbb D)).
\]
Moreover, \(L^\infty(\mathbb D)*_\pi\mathcal S_1(A^2(\mathbb D))\)
is a proper dense, and hence nonclosed, subspace of
\(\mathcal T(L^\infty(\mathbb D))\).
\end{corollary}

Indeed, by \cite[Theorem 6.11]{DawsonDewageMitkovskiOlafsson}, norm
convergence \(Q_m(R)\to R\) holds for every
\[
R\in L^\infty(\mathbb D)*_\pi\mathcal S_1(A^2(\mathbb D)).
\]
The operator supplied by the previous theorem therefore does not belong to this
class. On the other hand, this class contains every Toeplitz operator
\(T_a=a*_\pi\Phi\), and Toeplitz operators are norm dense in
\(\mathcal T(L^\infty(\mathbb D))\) by Xia's theorem~\cite{Xia}.

The mechanism is a moving matrix edge.  Although $Q_m$ preserves each matrix
diagonal, the relevant diagonal changes with $m$: the input column is $q=m$,
while the input row remains in a fixed finite block.  Along this edge, exact
Bergman coefficients converge to a weighted Pascal kernel, so one rank-one
column produces a two-parameter weighted Hankel matrix. More explicitly, on the relevant moving edge we set
\[
p=r+u,\qquad q=m,\qquad k=m-u.
\]
The diagonal-conservation law then gives an output row \(r\), while the
corresponding coefficients satisfy
\[
c_m(r+u,m;m-u)\longrightarrow K(r+u,r)=W(r,u).
\]
Thus the row index \(r\) and the offset \(u\) become the two indices of the
limiting Pascal-weighted Hankel matrix.

The proof has three steps.  We first compute the exact coefficients and
establish the moving-edge limit
\[
   c_m(p,m;m+r-p)\longrightarrow K(p,r).
\]
We then prove a standalone Pascal--Hankel instability theorem using parabolic
blocks and quadratic chirps.  Classical Schatten-class Hankel theory usually
studies the unweighted matrix $[a_{r+u}]$ through regularity of its symbol; see,
for example, \cite{Peller}.  The weight arising here is strongly damping but
has a critical parabolic window, and its trace-class behavior is different.
Finally, trace duality and a transpose identity transfer the finite-section
growth to the exact maps and give
\[
   \sup_m\|Q_m:\Kop(\A)\longrightarrow\Bop(\A)\|=\infty.
\]
The Uniform Boundedness Principle then produces a single compact operator on
which the family is unbounded.

Sections~\ref{sec:coefficients}--\ref{sec:transfer} contain the proof.
Section~\ref{sec:radial} explains why radial approximation does not see the
obstruction,.

\section{Exact matrix coefficients and the moving edge}
\label{sec:coefficients}

For $m\in\Nzero$, the higher-order Berezin transform can be written in the
normalization
\[
 B_m(S)(z)
  =(m+1)(1-|z|^2)^{m+2}
    \sum_{j=0}^m(-1)^j\binom mj
       \big\langle S v_{j,z}^{(m)},v_{j,z}^{(m)}\big\rangle,
\]
where
\[
   v_{j,z}^{(m)}(w)=w^jK_z^{(m)}(w),
   \qquad
   K_z^{(m)}(w)=\frac{1}{(1-\overline zw)^{m+2}}.
\]
Equivalently, in the quantum-harmonic-analysis notation recalled below, it is
generated by the radial finite-rank operator
\[
   \Phi_m=(m+1)\sum_{j=0}^m(-1)^j\binom mj\frac{1}{j+1}E_{jj}
\]

\begin{lemma}\label{lem:exact-coeff}
Let $p,q,k\in\Nzero$ and put $r=k+p-q$.  If $r<0$, then
$Q_m(E_{pq})e_k=0$.  If $r\geq0$, then
\[
   Q_m(E_{pq})e_k=c_m(p,q;k)e_r,
\]
where
\begin{align}
c_m(p,q;k)
   &=(m+1)\sqrt{\frac{(k+1)(r+1)}{(p+1)(q+1)}} \notag\\
   &\quad\times
   \sum_{j=0}^{\min\{m,p,q\}}(-1)^j
       \binom mj
       \binom{m+1+p-j}{p-j} \notag\\
   &\qquad\qquad\times
       \binom{m+1+q-j}{q-j}
       B(p-j+k+1,m+3).                         \label{eq:exact-coeff}
\end{align}
In particular, an input entry in position $(p,q)$ can contribute only to
output positions $(r,k)$ satisfying $p+k=q+r$.  Thus every matrix diagonal is
invariant under $Q_m$.
\end{lemma}

\begin{proof}
Expanding $K_z^{(m)}$ in powers of $w$ gives, for $s\geq j$,
\[
   \big\langle v_{j,z}^{(m)},e_s\big\rangle
     =\frac{1}{\sqrt{s+1}}
       \binom{m+1+s-j}{s-j}\overline z^{\,s-j},
\]
and the coefficient is zero for $s<j$.  Substituting this into the definition
of $B_m(E_{pq})$ and then into the Toeplitz matrix coefficient gives an angular
integral.  It vanishes unless
\[
   p-j+k=q-j+r,
\]
which is the conservation law above.  The remaining radial integral is
\[
   \int_0^1 t^{p-j+k}(1-t)^{m+2}\,dt
      =B(p-j+k+1,m+3),
\]
and \eqref{eq:exact-coeff} follows.
\end{proof}

We next record the trace-transpose property used in the transfer argument.  To
make the calculation explicit, let $U_z$ be the standard change-of-variable
unitary on $\A$.  If $R$ is trace class and $S$ is bounded, write
\[
   (S\ast_\pi R)(z)=\Tr(SU_zRU_z).
\]
If $h\in L^\infty(\D)$, set
\[
   h\ast_\pi R
      =\int_{\D}h(z)U_zRU_z\,d\lambda(z),
   \qquad
   d\lambda(z)=\frac{dA(z)}{(1-|z|^2)^2},
\]
where the integral is taken weakly.  With $\Phi=1\otimes1$, one has
\[
 B_0(S)=S\ast_\pi\Phi,
 \qquad T_h=h\ast_\pi\Phi,
 \qquad B_m(S)=S\ast_\pi\Phi_m.
\]
The radial interchange identity for trace-class generators gives
\begin{equation}\label{eq:radial-interchange}
   (S\ast_\pi\Phi_m)\ast_\pi\Phi
      =(S\ast_\pi\Phi)\ast_\pi\Phi_m;
\end{equation}
see \cite[Lemma~3.11, Lemma~6.8, and Appendix~A.10]
{DawsonDewageMitkovskiOlafsson}.  These definitions and
\eqref{eq:radial-interchange} are the only pieces of convolution formalism used
below.

\begin{lemma}\label{lem:trace-transpose}
For finite-rank operators $A,S$,
\[
   \Tr(Q_m(S)A)=\Tr(SQ_m(A)).
\]
Equivalently, whenever $r=k+p-q\geq0$,
\[
   c_m(p,q;k)=c_m(k,r;p).
\]
\end{lemma}

\begin{proof}
By \eqref{eq:radial-interchange},
\[
   Q_m(S)=(S\ast_\pi\Phi)\ast_\pi\Phi_m.
\]
The definition of the weak integral and cyclicity of the trace give
\begin{align*}
\Tr(Q_m(S)A)
 &=\int_{\D}(S\ast_\pi\Phi)(z)
      \Tr(U_z\Phi_mU_zA)\,d\lambda(z)\\
 &=\int_{\D}(S\ast_\pi\Phi)(z)
      (A\ast_\pi\Phi_m)(z)\,d\lambda(z)\\
 &=\Tr\!\left(S\big((A\ast_\pi\Phi_m)\ast_\pi\Phi\big)\right)
  =\Tr(SQ_m(A)).
\end{align*}
The integrals above are absolutely convergent. Indeed, the standard
convolution estimates (see~\cite[Appendix A]{DawsonDewageMitkovskiOlafsson}) give
\[
\|S *_\pi \Phi\|_\infty
   \leq \|S\|\,\|\Phi\|_1,
\qquad
\|A *_\pi \Phi_m\|_{L^1(d\lambda)}
   \leq \|A\|_1\,\|\Phi_m\|_1.
\]
Hence the product appearing in the preceding integral belongs to
\(L^1(d\lambda)\).  Taking $S=E_{pq}$ and
$A=E_{kr}$ gives the coefficient symmetry.
\end{proof}

The counterexample comes from a moving matrix edge: keep the input row $p$ and
output row $r$ fixed, put the input column at $q=m$, and let $m$ tend to
infinity.

\begin{lemma}\label{lem:finite-difference}
Fix $p,r\in\Nzero$.  Define
\[
   g_m(x)=(x+1)_{p+1}
       \frac{\Gamma(m+x+2)\Gamma(x+r+1)}
            {\Gamma(x+1)\Gamma(m+x+r+4)},
\]
where $(x+1)_{p+1}=(x+1)\cdots(x+p+1)$, and let
$\nabla f(x)=f(x)-f(x-1)$.  Then, for all sufficiently large $m$,
\[
 c_m(p,m;m+r-p)
   =\sqrt{\frac{(m+r-p+1)(r+1)}{(p+1)(m+1)}}
      \frac{m+2}{p!}\,\nabla^p g_m(m).
\]
\end{lemma}

\begin{proof}
Set $q=m$ and $k=m+r-p$ in \eqref{eq:exact-coeff}.  Since $p$ is fixed, the
sum ends at $j=p$.  After simplifying the beta function and binomial
coefficients, the $j$th summand is
\[
   (m+2)\frac{(-1)^j}{j!(p-j)!}\,g_m(m-j).
\]
Using $1/(j!(p-j)!)=\binom pj/p!$ gives the stated backward difference.
\end{proof}

\begin{lemma}\label{lem:pascal-kernel}
For every fixed $p,r\in\Nzero$,
\[
   \lim_{m\to\infty}c_m(p,m;m+r-p)=K(p,r),
\]
where
\[
   K(p,r)=
   \frac{(p+r+1)!}
        {2^{p+r+2}p!r!\sqrt{(p+1)(r+1)}}.
\]
In fact,
\[
   c_m(p,m;m+r-p)=K(p,r)+O_{p,r}(m^{-1}).
\]
\end{lemma}

\begin{proof}
Put $G_m(y)=m^{1-p}g_m(my)$.  Because $p,r$ are nonnegative integers, the
gamma ratios reduce exactly to finite products:
\[
   G_m(y)=
   \frac{
      \prod_{a=1}^{p+1}\left(y+\frac am\right)
      \prod_{b=1}^{r}\left(y+\frac bm\right)}
   {
      \prod_{c=2}^{r+3}\left(1+y+\frac cm\right)}.
\]
Hence, on $[1/2,3/2]$,
\[
   G_m\longrightarrow
   \phi(y):=\frac{y^{p+r+1}}{(1+y)^{r+2}}
   \quad\text{in }C^p,
\]
with an $O_{p,r}(m^{-1})$ error in $C^p$.

Let $h=1/m$ and write $\nabla_hf(y)=f(y)-f(y-h)$.  The exact identity
\[
   \frac{\nabla_h^pf(1)}{h^p}
      =\int_{[0,1]^p}
        f^{(p)}\!\left(1-h(t_1+\cdots+t_p)\right)
        \,dt_1\cdots dt_p
\]
shows that
\[
   \frac{\nabla_h^pG_m(1)}{h^p}\longrightarrow\phi^{(p)}(1).
\]
Since
\[
   (m+2)\nabla^pg_m(m)
    =\frac{m+2}{m}\frac{\nabla_h^pG_m(1)}{h^p},
\]
we obtain $(m+2)\nabla^pg_m(m)\to\phi^{(p)}(1)$.  Finally,
\[
   \frac{\phi^{(p)}(1)}{p!}
    =[t^p]\frac{(1+t)^{p+r+1}}{(2+t)^{r+2}}
    =\frac1{2^{p+r+2}}\binom{p+r+1}{p},
\]
and therefore
\[
   \phi^{(p)}(1)=\frac{(p+r+1)!}{2^{p+r+2}(r+1)!}.
\]
The square-root factor in Lemma~\ref{lem:finite-difference} tends to
$\sqrt{(r+1)/(p+1)}$.  Combining the two limits gives the formula for
$K(p,r)$.  The $O_{p,r}(m^{-1})$ statement follows from the $C^p$ product
estimate above.
\end{proof}

\begin{remark}\label{rem:order-limits}
The order of limits is important.  Lemma~\ref{lem:pascal-kernel} is a
fixed-index statement.  Later we first fix a finite block and then let
$m\to\infty$; no uniformity in growing $p$ or $r$ is required.
\end{remark}

\section{A Pascal--Hankel instability theorem}
\label{sec:pascal-hankel}

The matrix appearing at the moving edge can be written without Bergman
notation.  Define
\[
 W(r,u)
   :=2^{-(2r+u+2)}\binom{2r+u+1}{r}
      \sqrt{\frac{r+u+1}{r+1}},
   \qquad r,u\geq0.
\]
Notice that $W(r,u)=K(r+u,r)$.  For a sequence $a=(a_n)_{n\geq0}$, consider
the Pascal-weighted Hankel matrix
\[
   \mathcal P(a)(r,u)=a_{r+u}W(r,u).
\]
For finite $I,U\subset\Nzero$, write $P_I$ and $P_U$ for the corresponding
coordinate projections.  Rectangular finite matrices are regarded as
operators on $\ell^2(\Nzero)$ by extending them by zero.

\begin{theorem}\label{thm:pascal-hankel}
There exists $a\in\ell^2(\Nzero)$ for which the trace norms of the rectangular
coordinate compressions of $\mathcal P(a)$ are unbounded.  Equivalently,
\[
   \sup_{\substack{I,U\subset\Nzero\\I,U\ \mathrm{finite}}}
   \big\|[a_{r+u}W(r,u)]_{r\in I,\,u\in U}\big\|_1=\infty.
\]
In particular, the formal correspondence \(a\mapsto P(a)\), initially
specified through its finite coordinate compressions, does not extend to a
bounded map
\[
\ell^2(\mathbb N_0)\longrightarrow\mathcal S_1.
\]
\end{theorem}

The proof uses one elementary finite-dimensional observation: a rectangular
matrix with many orthogonal columns can have small operator norm but order-one
trace norm even when all its entries are small.

Fix a scale $R=16^j$, $j\geq1$, and define the parabolic rectangle
\[
 I_R=\{R,R+1,\ldots,2R-1\},
 \qquad
 U_R=\{\sqrt R,\sqrt R+1,\ldots,2\sqrt R-1\}.
\]
Put
\[
   \Pi_R=I_R+U_R
      =\{R+\sqrt R,\ldots,2R+2\sqrt R-2\}.
\]
We first remove the normalization naturally carried by an $\ell^2$ coefficient
sequence.  Set
\[
 w(r,u):=\frac{W(r,u)}{\sqrt{r+u+1}}
      =2^{-(2r+u+2)}\frac1{\sqrt{r+1}}
         \binom{2r+u+1}{r}.
\]

\begin{lemma}\label{lem:flatness}
There are absolute constants $0<c<C<\infty$ such that, for every $R=16^j$
and every $(r,u)\in I_R\times U_R$,
\[
   \frac cR\leq w(r,u)\leq\frac CR.
\]
Consequently,
\[
   \sum_{r\in I_R}\sum_{u\in U_R}w(r,u)^2\asymp R^{-1/2}.
\]
\end{lemma}

\begin{proof}
Let $N=2r+u+1$.  Then
\[
   w(r,u)=\frac1{2\sqrt{r+1}}\,2^{-N}\binom Nr.
\]
On $I_R\times U_R$ we have $N\asymp R$, while
\[
   \left|r-\frac N2\right|=\frac{u+1}{2}\asymp\sqrt R.
\]
The uniform local central-limit estimate for the symmetric binomial
distribution gives absolute constants \(c_0,C_0>0\) such that
\[
c_0R^{-1/2}
\leq
2^{-N}\binom{N}{r}
\leq
C_0R^{-1/2}
\]
for every \((r,u)\in I_R\times U_R\).  Since $r\asymp R$, this yields
$w(r,u)\asymp R^{-1}$.  The rectangle has $R\sqrt R$ entries, giving the
Hilbert--Schmidt estimate.
\end{proof}

\begin{lemma}\label{lem:schur}
On $I_R\times U_R$, the scalar matrix $(Rw(r,u))$ has Schur-multiplier norm
bounded by an absolute constant, uniformly in $R=16^j$.
\end{lemma}

\begin{proof}
With $N=2r+u+1$, the Fourier-coefficient formula for $\cos(t/2)^N$ gives
\[
 Rw(r,u)=\frac{R}{4\pi\sqrt{r+1}}
   \int_{-\pi}^{\pi}
      e^{i(u+1)t/2}\cos(t/2)^{2r+u+1}\,dt.
\]
Factor the integrand as
\[
 \left[\frac{R}{\sqrt{r+1}}\cos(t/2)^{2r}\right]
 \left[e^{i(u+1)t/2}\cos(t/2)^{u+1}\right].
\]
For fixed $t$, this is a rank-one scalar matrix in $(r,u)$; its
Schur-multiplier norm is bounded by the product of the two suprema.  Since
$r\geq R$, one has $R/\sqrt{r+1}\leq\sqrt R$, and the second factor has
modulus at most one.  Integrating the rank-one bounds therefore gives
\[
   \|(Rw(r,u))\|_{\mathrm{Schur}}
      \leq\frac{\sqrt R}{4\pi}
         \int_{-\pi}^{\pi}|\cos(t/2)|^{2R}\,dt,
\]
which is bounded uniformly in $R$ by the standard Gaussian estimate near
$t=0$.
\end{proof}

The oscillation is supplied by a quadratic chirp.

\begin{lemma}\label{lem:chirp}
Define
\[
   H_R(r,u)=\exp\!\left(\frac{2\pi i(r+u)^2}{R}\right),
   \qquad (r,u)\in I_R\times U_R.
\]
Then $H_R^*H_R=R\id$ and $\|H_R\|=\sqrt R$.
\end{lemma}

\begin{proof}
For $u,v\in U_R$,
\begin{align*}
(H_R^*H_R)(u,v)
 &=\sum_{r=R}^{2R-1}
    \exp\!\left(\frac{2\pi i((r+v)^2-(r+u)^2)}R\right)\\
 &=\exp\!\left(\frac{2\pi i(v^2-u^2)}R\right)
    \sum_{r=R}^{2R-1}
      \exp\!\left(\frac{2\pi i\,2(v-u)r}R\right).
\end{align*}
If $u=v$, the sum is $R$.  If $u\neq v$, then
$0<2|u-v|<2\sqrt R<R$, so the last sum is a complete nontrivial geometric
sum and equals zero.
\end{proof}

For $(r,u)\in I_R\times U_R$, set
\[
   F_R(r,u)=w(r,u)H_R(r,u).
\]

\begin{proposition}\label{prop:one-unit}
There is an absolute $c>0$ such that
\[
   \|F_R\|_1\geq c
\]
for every $R=16^j$.
\end{proposition}

\begin{proof} Since
\[
F_R=\frac{1}{R}\bigl((Rw)\circ H_R\bigr),
\]
where \(\circ\) denotes entrywise multiplication,
Lemmas~\ref{lem:schur} and \ref{lem:chirp} give,
\[
   \|F_R\|\leq\frac CR\|H_R\|\leq\frac C{\sqrt R}.
\]
For every finite matrix $X$,
$\|X\|_1\geq\|X\|_{\HS}^2/\|X\|$.  Applying this to $F_R$ gives the result.
\end{proof}

We now accumulate the one-scale units.  Let
\[
 R_j=16^j,
 \qquad
 I^{(J)}=\bigcup_{j=1}^J I_{R_j},
 \qquad
 U^{(J)}=\bigcup_{j=1}^J U_{R_j}.
\]
The sumsets $\Pi_{R_j}$ are pairwise disjoint.  Define
\[
 b_p^{(J)}=
 \begin{cases}
 \exp(2\pi i p^2/R_j),&p\in\Pi_{R_j},\quad1\leq j\leq J,\\
 0,&\text{otherwise},
 \end{cases}
\]
and
\[
   N_J=\sum_{p\geq0}\frac{|b_p^{(J)}|^2}{p+1}\asymp J.
\]
Indeed, each $\Pi_{R_j}$ has cardinality comparable to $R_j$, and
$p+1\asymp R_j$ on that interval.  Set
\[
   a_p^{(J)}=\frac{b_p^{(J)}}{\sqrt{N_J}\sqrt{p+1}}.
\]
Then $\|a^{(J)}\|_2=1$.

For each $J$, define the linear finite-section map
\[
 \mathcal P_J:\ell^2(\Nzero)\longrightarrow\Sone,
 \qquad
 \mathcal P_J(a)=[a_{r+u}W(r,u)]_{r\in I^{(J)},\,u\in U^{(J)}},
\]
again extending the finite matrix by zero.

\begin{proposition}\label{prop:multiscale}
There is an absolute $c>0$ such that
\[
   \|\mathcal P_J(a^{(J)})\|_1\geq c\sqrt J.
\]
Consequently,
$\|\mathcal P_J:\ell^2(\Nzero)\to\Sone\|\geq c\sqrt J$.
\end{proposition}

\begin{proof}
On the matched block $I_{R_j}\times U_{R_j}$,
\[
   \mathcal P_J(a^{(J)})(r,u)
      =\frac1{\sqrt{N_J}}F_{R_j}(r,u).
\]
Let
\[
\mathcal E_J(X)=\sum_{j=1}^J P_jXQ_j,
\]
where \(P_j\) and \(Q_j\) are the coordinate projections onto
\(I_{R_j}\) and \(U_{R_j}\), respectively. This map is trace-norm
contractive. Indeed,
\[
\mathcal E_J(X)
=
\frac{1}{2\pi}\int_0^{2\pi}
\left(\sum_{j=1}^J e^{ijt}P_j\right)
X
\left(\sum_{j=1}^J e^{ijt}Q_j\right)^*
\,dt,
\]
and both sums of phased orthogonal projections have norm one.

Moreover, the sets \(I_{R_j}\) are pairwise disjoint, as are the sets
\(U_{R_j}\). Hence \(\mathcal E_J(X)\) is a direct sum of its matched
blocks and
\[
\|\mathcal E_J(X)\|_1
=
\sum_{j=1}^J \|P_jXQ_j\|_1.
\]  

Therefore,
using Proposition~\ref{prop:one-unit},
\[
\begin{aligned}
\|P_J(a^{(J)})\|_1
&\geq \|\mathcal E_J(P_J(a^{(J)}))\|_1 \\
&= \frac{1}{\sqrt{N_J}}
   \sum_{j=1}^J \|F_{R_j}\|_1 \\
&\gtrsim \frac{J}{\sqrt{N_J}}
 \asymp \sqrt{J}.
\end{aligned}
\]
Since $\|a^{(J)}\|_2=1$, the operator-norm conclusion follows.
\end{proof}

\begin{proof}[Proof of Theorem~\ref{thm:pascal-hankel}]
Put
\[
F_J=I^{(J)}+U^{(J)}.
\]
Since \(F_J\) is finite, \(P_J\) depends only on the coordinates
\((a_p)_{p\in F_J}\). Thus \(P_J\) factors through the coordinate
restriction
\[
\ell^2(\mathbb N_0)\longrightarrow \ell^2(F_J),
\]
and is therefore a bounded linear map
\[
P_J:\ell^2(\mathbb N_0)\longrightarrow\mathcal S_1.
\]  Proposition~\ref{prop:multiscale} gives
$\sup_J\|\mathcal P_J\|=\infty$.  By the Uniform Boundedness Principle, there
exists a single $a\in\ell^2$ such that
\[
   \sup_J\|\mathcal P_J(a)\|_1=\infty.
\]
After normalization, we may assume $\|a\|_2=1$.
\end{proof}

\begin{remark}
If the multiplier $W$ were absent, the same chirp idea would be much stronger.
On an $R\times cR$ rectangle, the normalized sequence
$a_p\asymp R^{-1/2}e^{2\pi ip^2/R}$ produces a partial Fourier matrix with
$\asymp R$ orthogonal columns of norm $\asymp1$, hence trace norm
$\asymp R$.  The Pascal multiplier forces the useful transverse width down to
$\sqrt R$ and contributes an additional factor $R^{-1/2}$, leaving only
order-one trace norm per scale.  The multiscale argument recovers divergence
from this critical loss.
\end{remark}

\section{Transfer to the exact Bergman maps}
\label{sec:transfer}

Fix the sequence $a=(a_p)\in\ell^2$ given by
Theorem~\ref{thm:pascal-hankel}, normalized by $\|a\|_2=1$, and put
\[
   x=\sum_{p\geq0}a_pe_p\in\A.
\]

The purpose of this section is to transfer the trace-norm growth of finite
compressions of \(Q_m(A_m)\) into operator-norm growth of the maps \(Q_m\).

The rank-one operators used below are moving test operators. The single
compact counterexample is obtained only at the end, by the
Uniform Boundedness Principle. For every $m$, define the rank-one operator
\[
   A_m=x\otimes e_m.
\]
Then $\|A_m\|_1=1$.

For fixed $J$ and $m>\max U^{(J)}$, define
\[
 C_{J,m}=\{m-u:u\in U^{(J)}\},
 \qquad
 Z_{J,m}=P_{I^{(J)}}Q_m(A_m)P_{C_{J,m}}.
\]
Identify $\ell^2(U^{(J)})$ with $\ell^2(C_{J,m})$ by
$e_u\mapsto e_{m-u}$.

\begin{lemma}\label{lem:exact-convergence}
For every fixed $J$,
\[
   \lim_{m\to\infty}\|Z_{J,m}-\mathcal P_J(a)\|_1=0.
\]
\end{lemma}

\begin{proof}
By the diagonal conservation law in Lemma~\ref{lem:exact-coeff}, the only
coefficient of the one-column operator
$A_m=\sum_{p\geq0}a_pE_{p,m}$ that can contribute to row $r$ and column
$m-u$ is $p=r+u$.  Hence
\[
   \langle Q_m(A_m)e_{m-u},e_r\rangle
       =a_{r+u}c_m(r+u,m;m-u).
\]
For fixed $J$, the pairs $(r,u)$ range over finite sets.
Lemma~\ref{lem:pascal-kernel}, applied with $p=r+u$, gives
\[
   c_m(r+u,m;m-u)\longrightarrow K(r+u,r)=W(r,u).
\]
Thus the finite matrices converge entrywise.  Their dimensions are fixed, so
they converge in every matrix norm, in particular in trace norm.
\end{proof}

Since $\sup_J\|\mathcal P_J(a)\|_1=\infty$, the lemma implies that, for every
$M>0$, there are $J$ and then $m$ sufficiently large such that
$\|Z_{J,m}\|_1>M$.

\begin{proposition}
\label{prop:unbounded-restrictions}
\[
   \sup_{m\geq0}
   \|Q_m:\Kop(\A)\longrightarrow\Bop(\A)\|=\infty.
\]
\end{proposition}

\begin{proof}
Choose $J,m$ with $\|Z_{J,m}\|_1>M$, and write the polar decomposition
$Z_{J,m}=U|Z_{J,m}|$.  Extend $V=U^*$ by zero outside the corresponding row
and column spaces.  Then $V$ is finite rank, $\|V\|=1$, and
\[
   V=P_{C_{J,m}}VP_{I^{(J)}}.
\]
Consequently, cyclicity of the trace and the support identities give
\begin{align*}
\Tr(VQ_m(A_m))
 &=\Tr\!\left(VP_{I^{(J)}}Q_m(A_m)P_{C_{J,m}}\right)\\
 &=\Tr(U^*Z_{J,m})=\|Z_{J,m}\|_1.
\end{align*}
By Lemma~\ref{lem:trace-transpose},
\[
   \Tr(Q_m(V)A_m)=\|Z_{J,m}\|_1.
\]
Since $\|A_m\|_1=1$,
\[
   \|Q_m(V)\|
      \geq|\Tr(Q_m(V)A_m)|
      =\|Z_{J,m}\|_1>M.
\]
As $M$ is arbitrary, the restrictions are not uniformly bounded.

For completeness, each fixed $Q_m$ is bounded on $\Bop(\A)$.  Indeed, the
defining formula gives
$\|B_m(S)\|_\infty\leq C_m\|S\|$, while
$\|T_f\|\leq\|f\|_\infty$.  Thus
$Q_m|_{\Kop(\A)}:\Kop(\A)\to\Bop(\A)$ is a bounded linear map between Banach
spaces.
\end{proof}

\begin{proof}[Proof of Theorem~\ref{thm:main}]
Proposition~\ref{prop:unbounded-restrictions} and the Uniform Boundedness
Principle give an $S\in\Kop(\A)$ such that
\[
   \sup_{m\geq0}\|Q_m(S)\|=\infty.
\]
Every finite initial set of indices contributes only finitely many finite
norms, so the same supremum is infinite on every tail.  We may therefore choose
strictly increasing $m_n$ such that
\[
   \|Q_{m_n}(S)\|\geq2^n.
\]
It follows that
\[
   \|Q_{m_n}(S)-S\|\geq2^n-\|S\|\longrightarrow\infty.
\]

It remains to note that every compact operator belongs to the full Toeplitz
algebra.  Let $W=T_z$.  In the basis $(e_n)$,
\[
   We_n=\sqrt{\frac{n+1}{n+2}}\,e_{n+1}.
\]
The commutator $[W^*,W]$ is the positive diagonal compact operator with
eigenvalues
\[
   \frac12,\ \frac1{2\cdot3},\ \frac1{3\cdot4},\ldots.
\]
The eigenvalue $1/2$ is isolated, so its spectral projection $E_{00}$ belongs
to $C^*(W)$.  Multiplying by powers of $W$ and $W^*$, and rescaling by the
nonzero weights, shows that every matrix unit $E_{pq}$ belongs to $C^*(W)$.
Therefore
\[
   \Kop(\A)\subset C^*(T_z)\subset\Top(L^\infty),
\]
and the compact operator $S$ above lies in the full Toeplitz algebra.
\end{proof}

\section{Why the radial theory does not see the obstruction}
\label{sec:radial}

The matrix calculation makes the contrast with the radial case transparent.
A radial operator is diagonal in the monomial basis and hence lives on the
main matrix diagonal.  The obstruction above is a moving-edge phenomenon: the
input column is $q=m$, while the relevant input row remains fixed on each
finite block, so the angular frequency $p-q$ tends to $-\infty$.

There is also a convolution explanation.  For arbitrary $S$,
\[
   Q_m(S)=(S\ast_\pi\Phi_m)\ast_\pi\Phi.
\]
If $S$ is radial, the radial interchange identity allows the factors to be
commuted and gives
\[
   Q_m(S)=\varphi_m\ast_\pi S,
   \qquad
   \varphi_m:=\Phi_m\ast_\pi\Phi.
\]
Here
\[
\varphi_m(z)=(m+1)(1-|z|^2)^{m+2}.
\]
Thus \(\varphi_m\geq 0\),
\[
\int_{\mathbb D}\varphi_m(z)\,d\lambda(z)=1,
\]
and \((\varphi_m)\) is a positive approximate identity.  Radial norm
convergence is therefore an ordinary approximate-identity phenomenon on the
norm-continuous radial orbit; see
\cite{BauerHerreraVasilevski,DawsonDewageMitkovskiOlafsson}.

The counterexample exploits precisely what disappears under radialization.
The maps $Q_m$ preserve every matrix diagonal, but there is no uniform control
over diagonals whose angular frequency escapes with $m$.  On the moving edge,
the exact coefficient tensor converges to the Pascal kernel and turns one
rank-one input column into a weighted Hankel matrix.  The Pascal weight has the
critical parabolic size
\[
   W(r,u)\asymp r^{-1/2}
   \qquad\text{for }u\asymp\sqrt r,
\]
and decays rapidly outside that window.  Thus each large scale retains an
order-one amount of trace norm, and separated scales accumulate through an
$\ell^1/\ell^2$ mismatch.

\paragraph{\bf AI disclosure statement} The GPT-5.6 Sol model with max effort, accessed through OpenAI’s Codex CLI, was used to construct the initial version of this counterexample.
More precisely, the model was used intermittently over an elapsed period of approximately 19.5 hours to analyze Suárez’s
higher-order Berezin approximation conjecture and search for a counterexample in full generality. The authors subsequently
independently checked and validated the construction, simplified several parts of the analysis, and rewrote the paper in
its entirety. They take full responsibility for all mathematical content of the paper, including every claim and argument
concerning the counterexample construction.

This work was supported in part by OpenAI API credits provided by Clemson University and administered by Clemson University Research Computing and Data (RCD).

\enlargethispage{2\baselineskip}

\end{document}